\documentclass[a4paper, reqno]{amsart}

\author{Nick Dewaele}
\date{}

\title{Which variables of a numerical problem cause ill-conditioning?}

\usepackage{amsmath, amsthm, amssymb, dsfont, hyperref, physics, graphicx}
\usepackage[capitalise]{cleveref}
\usepackage{xcolor}
\usepackage[
  backend=bibtex,
  doi=false,
  eprint=false,
  url=false,
  maxcitenames=3,
  style=alphabetic,
  isbn=false
]{biblatex}
\DeclareMathAlphabet{\mathpzc}{OT1}{pzc}{m}{it}

\newtheorem{theorem}{Theorem}
\newtheorem{proposition}[theorem]{Proposition}
\newtheorem{lemma}[theorem]{Lemma}

\newtheorem{remark}[theorem]{Remark}
\newtheorem{corollary}[theorem]{Corollary}

\theoremstyle{definition}
\newtheorem{definition}[theorem]{Definition}

\newcommand{\tspace}{{\mathcal{T}}}

\newcommand{\tuck}{{G_T}}
\newcommand{\core}{{\mathpzc{C}}}
\newcommand{\kron}{{\oslash}}

\renewcommand{\epsilon}{\varepsilon}

\DeclareMathOperator{\nullity}{null}
\DeclareMathOperator{\img}{Im}
\DeclareMathOperator*{\argmin}{arg\,min}

\begin{document}

\thanks{This research was partially funded by Internal Funds KU Leuven with project number C16/21/00}

\begin{abstract}
    We study a broad class of numerical problems that can be defined as the solution of a system of (nonlinear) equations for a subset of the dependent variables.
    Given a system of the form $F(x,y,z) = c$ with multivariate input $x$ and dependent variables $y$ and $z$, we define and give concrete expressions for the condition number of solving for a value of $y$ such that $F(x,y,z) = c$ for some unspecified $z$. This condition number can be used to determine which of the dependent variables of a numerical problem are the most ill-conditioned. We show how this can be used to explain the condition number of the problem of solving for all dependent variables, even if the solution is not unique. The concepts are illustrated with Tucker decomposition of tensors as an example problem.
\end{abstract}

\maketitle

\section{Introduction}

Numerical problems in finite dimension can often be formulated as the solution of a system of equations of the form $F(x,y) = 0$, where $x$ is the input variable, living in some input space $\mathcal{X}$, and the goal is to solve the equation for $y$ in some output space $\mathcal{Y}$. Quantifying the sensitivity of the solution with respect to the input is an essential aspect of numerical analysis.

If the solution $y$ is unique given the input $x$, the \emph{condition number} of the problem measures its sensitivity and is formally defined as follows \cite{Rice1966}. Suppose that $\mathcal{X}$ and $\mathcal{Y}$ are metric spaces with distances $d_{\mathcal X}$ and $d_{\mathcal{Y}}$, respectively and that $F(x,y) = 0$ at a particular pair $(x_0, y_0)$.
Then the (absolute) condition number at $(x_0, y_0)$ is 
\begin{equation}
    \label{eq: def cond limsup}
\kappa[F](x_0, y_0) :=
 \limsup_{\substack{x \to x_0 \\ F(x,y) = 0}} \frac{d_{\mathcal Y}(y_0, y)}{d_{\mathcal X}(x_0, x)}
.\end{equation}
Generalisations of this definition are used if multiple solutions may correspond to the same input; see \cite[Chapter 14]{Burgisser2013} for when the solution sets $\{y \,|\, F(x,y) = 0\}$ consist of isolated points and \cite{Dewaele2024} for when they are manifolds of non-zero dimension. In those cases, the value of $y$ used in \eqref{eq: def cond limsup} is the solution of $F(x,y) = 0$ that is closest to $y_0$.

Some numerical problems involve additional, \emph{latent} variables that are neither the input nor the solution. In those cases, we can model the problem as a system $F(x,y,z) = 0$ where $x$ is the input and a value of $y$ is desired such that $F(x,y,z) = 0$ for some $z$ which is not considered part of the solution.
For instance, suppose that $X \in \mathbb{R}^{m \times n}$ is a matrix that is assumed to have rank $k$ and one is interested in finding a matrix $Y \in \mathbb{R}^{m \times k}$ such that $\operatorname{span} X = \operatorname{span} Y$. We can model this problem as the system $X - YZ = 0$ where $Z \in \mathbb{R}^{k \times n}$.

It can be useful to have a theory of condition in the presence of latent variables. Informally, such a condition number should measure, for the worst perturbation to $x$, the change in $y$ such that $F(x,y,z) = 0$ for some $z$.
While we have not found the general form of this in the literature, it is surely known to the experts how the condition number would follow from \eqref{eq: def cond limsup} if $y$ is uniquely determined given $x$. Indeed, if $F(x,y,z) = 0$ has only one solution $y$ for every $x$, this solution is some function $y = H(x)$. Hence, the desired condition number is given by applying \eqref{eq: def cond limsup} to the equation $y - H(x) = 0$ (assuming that subtraction is defined over the output space $\mathcal{Y}$).

Recently, we showed the usefulness of an extended theory of condition that is applicable to \emph{(nonlinear) underdetermined systems}, which we define as systems $F(x,y) = 0$ that have infinitely many solutions $y$ for any input $x$ \cite{Dewaele2024}. Systems with exactly one solution for each input are called \emph{identifiable}.
Even when one is interested in an identifiable problem $F(x,y) = 0$, underdetermined systems nevertheless arise when a subset of the constraints (i.e.\ equations) in $F$ is taken out of consideration to study the remaining constraints separately. We say that a system $R(x,y) = 0$ is a \emph{relaxation} of $F(x,y) = 0$ if it can be defined by taking a subset of the equations from $F$. In other words, $R = P \circ F$ where $P$ is any projection. Under some conditions, specified below, underdetermined systems have a condition number that is consistent with the standard definition \eqref{eq: def cond limsup}, even though they are classically considered ill-posed in numerical linear algebra.

The utility of this augmented theory of condition comes from its added explanatory value: we showed that, at any valid input-output pair $(x_0, y_0)$, any system is at least as ill-conditioned as any of its relaxations. Suppose that we have two systems defined by $F$ and $R$ where $R$ defines a relaxation of $F$. Then solving $F(x,y) = 0$ for a given $x$ requires solving \emph{at least} $R(x,y) = 0$ in the sense that at least one solution of $R(x,y) = 0$ is needed. If solving $R(x,y) = 0$ is ill-conditioned by \cite{Dewaele2024}, so must be solving $F(x,y) = 0$ and the cause of the ill-conditioning of $F$ must already manifest itself in the constraints defining the (simpler) problem $R(x,y) = 0$.

To illustrate the above, consider the following system of equations with input $x \in \mathbb{R}$ and output $(y_1, y_2) \in \mathbb{R}^2$:
\begin{equation}
    \label{eq: example system introduction}
\begin{cases}
    y_1^2 + y_2^2 = (x-1)^{-2} \\
    y_1 - y_2 = 0
\end{cases}
.\end{equation}
The first equation says that $(y_1,y_2)$ lies on a circle with radius $|x - 1|^{-1}$. Since this radius changes rapidly for $x$ close to $1$, it is clear that solving for $(y_1,y_2)$ is ill-conditioned near $x = 1$, even without considering the second equation.

\subsection{Contributions}

The objective of this paper is to present a similar analysis at the level of the solution variables rather than the defining equations. That is, we compare the problem of solving a system $F(x,y,z) = 0$ for $y$ given $x$ to that of solving the same system for both $y$ and $z$ combined. We will conclude analogously that solving for the pair $(y,z)$ is at least as ill-conditioned as solving for $y$. Hence, whatever makes solving for $y$ ill-conditioned explains the ill-conditioning of solving for $(y,z)$.
For instance, write the dependent variables of \eqref{eq: example system introduction} in polar coordinates as $(y_1,y_2) = (y \cos z, y \sin z)$.
Then solving \eqref{eq: example system introduction} for $(y,z)$ is ill-conditioned because the radius $y$ is sensitive to $x$, although the angle $z$ is not. 

This kind of analysis arises naturally in numerical linear algebra. For example, suppose that we wish to compute an orthonormal basis of the column space of some full-rank matrix $X \in \mathbb{R}^{m \times n}$ with $m > n$. One method is to take a singular value decomposition $X = U\Sigma V^T$ and keep the $U$ factor.
Although this is reliable in practice \cite{Golub2013}, it may seem like a bad idea since computing the singular vectors is ill-conditioned if any pair of singular values of $X$ is close together (and even \emph{ill-posed} if $X$ has reoccurring singular values) \cite{sunPerturbationAnalysisSingular1996,vannieuwenhoven2024condition}.
Our framework provides a new interpretation of why this does not cause any problems. For this problem, we do not need all three factors $(U,\Sigma,V)$, nor do we rely on the constraint that $\Sigma$ is diagonal. The condition number of solving for $U$ that takes this into account is independent of the singular value gaps. This assertion is formalised by \cref{prop: cond Tucker}, a more general statement about tensor decompositions.

The key contribution of this paper is a general framework formalising these examples, while covering the case where the solution $y$ is not necessarily unique given the input $x$. Therefore, the results can be combined with the aforementioned theory about relaxations.
The types of problems we will study are defined as follows.
\begin{definition}
    \label{def: CREP}
    Let $F\colon \mathcal X \times \mathcal Y \times \mathcal Z \to \mathcal W$ be a smooth map between smooth manifolds $\mathcal{X},\mathcal{Y},\mathcal{Z},$ and $\mathcal{W}$ and let $c$ be a constant in the image of $F$. Let $DF$ be the total derivative (Jacobian) of $F$. Suppose that 
    \begin{itemize}
        \item $\rank DF(x,y,z) = \rank \pdv{(y,z)}F(x,y,z) = r$ everywhere for some $r \in \mathbb{N}$,
        \item $\rank \pdv{z}F(x,y,z) = k$ everywhere for some $k \in \mathbb{N}.$
    \end{itemize}
    Then the equation $F(x,y,z) = c$ is a \emph{constant-rank elimination problem (CREP)}. The spaces $\mathcal X, \mathcal Y,$ and $\mathcal Z$ are the \emph{input}, \emph{output}, and \emph{latent} space, respectively. 
\end{definition}
This problem class can be interpreted as follows: the nullity of the partial derivative of $F$ with respect to some variable gives the number of degrees of freedom in that variable. Hence, the constant-rank assumptions above say that, given an input $x$ that may be perturbed in any direction within $\mathcal X$, there is always a fixed number of degrees of freedom in $(y,z)$ and in $z$, though not necessarily the same number for both. \Cref{prop: geometric characterisation of FCRE} expresses this more formally.
For example, if $\mathcal X$ is the set of real $m \times n$ matrices of rank $k$, $\mathcal Y = \mathbb{R}^{m \times k}$, and $\mathcal Z = \mathbb{R}^{k \times n}$, then the equation $F(X,Y,Z) := X - YZ = 0$ is a CREP. The number of degrees of freedom is $k^2$ for both $(Y,Z)$ and $Z$ since, if $(Y,Z)$ solves $X - YZ = 0$, then all possible solutions are exactly $\{(YG,G^{-1}Z) \,\vert\,G \in \mathbb{R}^{k \times k}, \det G \ne 0\}$.

The main theorem of this article makes it possible to assign a condition number to CREPs. The following statement is an abridged version of it.

\begin{theorem}
    \label{thm: main theorem informal}
    Let $F(x,y,z) = c$ be a CREP with $F\colon \mathcal X \times \mathcal Y \times \mathcal Z \to \mathcal W$ and let $(x_0, y_0, z_0)$ be a solution to the CREP. Assume that $\mathcal{Y}$ has a Riemannian metric with an induced distance function $d_{\mathcal Y}$. Then there exists a neighbourhood $\widehat{\mathcal Z} \subseteq \mathcal Z$ of $z_0$ so that the \emph{canonical solution map} 
    \begin{align*}
        H\colon \mathcal X &\to \mathcal Y \\
	    x &\mapsto \argmin_{\substack{y \text{ s.t. } \exists z \in \widehat{\mathcal Z}\colon\\ F(x,y,z) = c}} d_{\mathcal Y}(y_0, y)
    \end{align*}
    is well-defined and smooth around $x_0$.
\end{theorem}

This result is proved in \cref{sec: geometry and proofs}.
\Cref{fig: CREP solution sets} is a visualisation of the theorem, for the specific case where $\pdv{F}{z}$ has full rank. If $\dim \mathcal Z > 1$ and $\pdv{F}{z}$ does not have full rank, the solution sets $\{(y,z) \,|\, F(x,y,z) = c\}$ have a higher dimension than their projections onto $\mathcal{Y}$.

\begin{figure}[t]
    \centering
    \includegraphics[width=\textwidth]{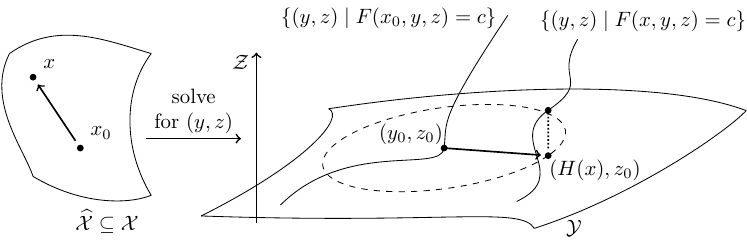}    
    \caption[Solution sets of a CREP and the canonical solution map $H$.]{Solution sets of a CREP and the canonical solution map $H$. Given an exact and perturbed input $x_0$ and $x$, the solution sets for both inputs are subsets of $\mathcal Y \times \mathcal Z$ (in this case, curves).
    For sufficiently small neighbourhoods $\widehat{\mathcal Z} \subseteq \mathcal Z$ of $z_0$, the solution map can be visualised as follows. For a time parameter $t \ge 0$, define the cylinder $C(t) := B(t) \times \widehat{\mathcal Z}$ where $B(t)$ is the closed disc around $y_0$ of radius $t$. Let $t$ increase from $0$ until $C(t)$ touches the solution set of $x$ at some point $(y,z)$. The unique $y$-coordinate of this point is $H(x)$ by definition.}
    \label{fig: CREP solution sets}
\end{figure}

\begin{definition}
    \label{def: cond CREP}
    Let $F(x,y,z) = c$ be a CREP with a particular solution $(x_0, y_0, z_0)$ and suppose that the input and output space have a Riemannian metric. Let $H$ and $\widehat{\mathcal Z}$ be defined by \cref{thm: main theorem informal}. Then the \emph{condition number of $y$ at $(x_0, y_0, z_0)$} is
    $$
    \kappa_{x \mapsto y}[F](x_0, y_0, z_0) := \limsup_{x \to x_0} \frac{d_{\mathcal Y}(y_0, H(x))}{d_{\mathcal X}(x_0, x)}
    =
    \limsup_{x \to x_0} \min_{\substack{F(x,y,z) = c\\z \in \widehat{\mathcal Z}}} \frac{d_{\mathcal Y}(y_0, y)}{d_{\mathcal X}(x_0, x)}
    ,$$
	where $d_{\mathcal X}$ and $d_{\mathcal Y}$ are the distances in $\mathcal X$ and $\mathcal Y$, respectively. The condition number of $z$, written as $\kappa_{x \mapsto z}$, is defined analogously by reversing the roles of $y$ and $z$. Finally, the condition number $\kappa_{x \mapsto (y,z)}$ is defined by applying the above to the equivalent CREP whose input, output, and latent space are $\mathcal X$, $\mathcal Y \times \mathcal Z$ (with the product metric), and the trivial set $\{0\}$, respectively. 
\end{definition}

A CREP whose latent space is $\mathcal Z$ is a singleton has effectively no latent variables. If this is the case and there is a unique solution $y$ for every input $x$, then by definition, $\kappa_{x \mapsto y}[F](x_0,y_0,z_0)$ is the condition number from \eqref{eq: def cond limsup} as it is conventionally understood. That is, \cref{def: cond CREP} reduces to \eqref{eq: def cond limsup} for identifiable problems.

For a perturbed input $x$, the condition number bounds the optimal forward error as
\begin{equation}
    \label{eq: 1st order error bound CREP}
\min_{
\substack{y \in \mathcal Y,\, z \in \widehat{\mathcal Z} \\ F(x,y,z) = c}
} d_{\mathcal Y}(y_0, y)
\le 
\kappa_{x \mapsto y}[F](x_0, y_0, z_0)\cdot d_{\mathcal X}(x_0, x) + o(d_{\mathcal X}(x_0, x))
\,\,\text{as}\,\, x \to x_0,
\end{equation}
where $\widehat{\mathcal Z}$ is some neighbourhood of $z_0$. The left-hand side can be interpreted as the optimal forward error $d_{\mathcal Y}(y_0, y)$ that can be attained with a solution $(y,z)$ close to $(y_0, z_0)$. It should be noted that closeness to $z_0$ is defined topologically rather than metrically and that a distance on $\mathcal{Z}$ is not required to define the condition number. 
That is, \cref{def: cond CREP} would give the same condition number for any metric on the manifold $\mathcal Z$.
Whilst this may seem unintuitive, this is ultimately because the condition number is a local (infinitesimal) property of the CREP and because measuring errors in the $z$-coordinate is irrelevant by assumption. 
Visually, we can picture the solution 
curve of $F(x,y,z) = c$ on \cref{fig: CREP solution sets} as merging into the solution curve of $F(x_0,y,z) = c$ as $x$ approaches $x_0$.
This explains why the condition number does not depend on the size of the neighbourhood $\widehat{\mathcal Z}$.
The distance we keep track of (i.e.\ the left-hand side of \eqref{eq: 1st order error bound CREP}) is measured only in the $y$-coordinate.

In \cref{sec: compute cond}, we will give several expressions that can be used to compute the condition number of a CREP. Loosely speaking, $\kappa_{x \mapsto y}[F](x_0,y_0,z_0)$ is the spectral norm of the linear map that takes any vector $\Delta x$ to $\Delta y$ such that $\norm{\Delta y}$ is minimised over all solutions $(\Delta y, \Delta z)$ of the linear system
$$
\pdv{F}{x}(x_0,y_0,z_0) \Delta x + \pdv{F}{y}(x_0,y_0,z_0) \Delta y + \pdv{F}{z}(x_0,y_0,z_0) \Delta z = 0,
$$
where the derivatives are evaluated at $(x_0,y_0,z_0)$. The rank assumptions in \cref{def: CREP} are made to ensure that this unique minimum-norm solution always exists.

The proposed condition number can measure which solution variables are the most sensitive to perturbations. The following statement expresses that solving for all dependent variables given the input is at least as ill-conditioned as solving for a subset of them. 

\begin{proposition}
    \label{prop: solving for some variables is easier}
    Let $F(x,y,z) = c$ be a CREP whose input, output, and latent space all have a Riemannian metric. At any solution $(x_0,y_0,z_0)$ to the CREP, we have
    \begin{equation}
        \label{eq: lower bound cond FCRE}
    \kappa_{x \mapsto y}[F](x_0,y_0,z_0) \le \kappa_{x \mapsto (y,z)}[F](x_0,y_0,z_0)
    .\end{equation}
\end{proposition}

This result, which is proven in \cref{sec: proofs}, can be used to determine which solution variables of a system of equations are the most sensitive to perturbations. If the ratio between right-hand side and the left-hand side of \eqref{eq: lower bound cond FCRE} is large, then $y$ is a dependent variable that is relatively insensitive to small changes in $x$ compared to the combination $(y,z)$.

It should be stressed that, since problems with an infinite number of solutions are allowed by \cref{def: CREP}, a CREP can be broken up in two ways: the first is to find out which (underdetermined) relaxations are ill-conditioned in order to identify the ill-conditioned constraints, and the second is to determine which solution variables are ill-conditioned by comparing the left and right-hand side of \eqref{eq: lower bound cond FCRE}. By combining both approaches, we can chisel down computational problems to the core variables and equations that make them ill-conditioned.

\subsection{Outline} 

The remainder of the paper is organised as follows. In \cref{sec: geometry and proofs}, we prove the main results, i.e.\ \cref{thm: main theorem informal} and \cref{prop: solving for some variables is easier}, by formulating CREPs as differential geometry problems. 
\Cref{sec: compute cond} shows how the condition number of a general CREP can be computed using numerical linear algebra.
\Cref{sec: Tucker} illustrates how the results above can be applied to an example problem, i.e.\ orthogonal Tucker decomposition of tensors. Concluding remarks are made in \cref{sec: conclusion}

\subsection{Notation}
For a point $p$ on a manifold $\mathcal M$, the tangent space is denoted as $\tspace_p \mathcal M$. Generic tangent vectors in this space are written as $\dot{p}$. The total derivative of a smooth map $F$ at $p$ is written as $DF(p)$. The derivative of $F$ in the direction $\dot{p}$ is $DF(p)[\dot p]$. Projections of a set $\mathcal M \times \mathcal N$ onto $\mathcal M$ and $\mathcal N$ are written as $\pi_{\mathcal M}$ and $\pi_{\mathcal N}$, respectively. $\mathcal X, \mathcal Y,$ and $\mathcal Z$ are smooth manifolds and $\mathcal X$ and $\mathcal Y$ are Riemannian. For any linear map $A$, $\nullity A$ is the nullity of $A$,  $\operatorname{span}A$ or $\operatorname{Im} A$ is the image, $\norm{A}$ is the operator norm (unless stated otherwise), and $\sigma_k(A)$ is the $k$th largest singular value. The set of $n \times n$ orthogonal matrices is $O(n)$.

\section{Geometric characterisation of the condition number}
\label{sec: geometry and proofs}

The goal of this section is to prove the two main theoretical results: \cref{thm: main theorem informal} and \cref{prop: solving for some variables is easier}. This will lead to an initial expression for the condition number, given by \cref{cor: norm solution map}.

The central idea in the proof is to leverage the theory of condition of \emph{feasible constant-rank equations (FCREs)} developed in \cite{Dewaele2024}. These can be defined as equations $\widetilde{F}(x,y) = c$ where adding a dummy variable $z := 0$ and defining
$F(x,y,z) := \widetilde{F}(x,y)$ gives a CREP $F(x,y,z) = c$.
In other words, FCREs are exactly CREPs without latent variables. The rank conditions in \cref{def: CREP} simplify to $\rank D\widetilde{F}(x,y) = \rank \pdv{y}\widetilde{F}(x,y) = r$ for FCREs.

The special case of \cref{thm: main theorem informal} for FCREs is \cite[Theorem 1.1]{Dewaele2024}, which we can summarise as follows. At any solution $(x_0, y_0)$ of the FCRE $\widetilde{F}(x,y) = c$, the derivative of $H(x) := \argmin_{\widetilde{F}(x,y) = c} d(y_0, y)$ at $x_0$ is
\begin{equation}
    \label{eq: deriv soln map FCRE}
DH(x_0) = - \left( \pdv{y} \widetilde{F}(x_0,y_0) \right)^\dagger \pdv{x}\widetilde{F}(x_0,y_0)
,\end{equation}
where $\cdot^\dagger$ is the Moore--Penrose inverse.
The spectral norm of $DH(x_0)$ is the condition number at $(x_0,y_0)$ in the sense of \cref{def: cond CREP}.

Our strategy for deriving the condition number of a CREP consists of two steps: first, in \cref{sec: geometric FCRE}, we find differential-geometric conditions that guarantee that projecting the graph of $F(x,y,z) = c$ onto the $x$ and $y$ variables gives an FCRE $\widetilde{F}(x,y) = 0$. Then, in \cref{sec: proofs}, we reformulate the expression for the condition number of the projected problem $\widetilde{F}(x,y) = 0$ in terms of the original CREP. 
The proof relies on the central concept in differential geometry of \emph{maps of constant rank}. A map $G\colon\mathcal M \to \mathcal N$ between smooth manifolds has constant rank $r$ if $\rank DG(p) = r$ for all $p \in \mathcal M$. Every level set of such a map $G$ is a smooth manifold whose tangent space is $\ker DG$ \cite[Theorem 5.12]{Lee2013}.

\subsection{Projections of CREPs}
\label{sec: geometric FCRE}

The following statement gives a way of recognising geometrically that a problem without latent variables satisfies the constant rank assumptions in \cref{def: CREP}. It also gives equivalent interpretations of these assumptions.

\begin{proposition}
    \label{prop: geometric characterisation of FCRE}

    Let $F: \mathcal X \times \mathcal Y \to \mathcal W$ be a smooth map of constant rank and let $\mathcal P$ be any non-empty level set of $F$.
    Then the following three statements are equivalent:
    \begin{enumerate}
        \item $
    \rank DF(x,y) = \rank \pdv{y}F(x,y)
    $ for all $(x,y)$,
        \item $\dim \mathcal P = \dim \mathcal X + \nullity \pdv{y}F(x,y)$ for all $(x,y)$,
        \item \label{itm: geometric definition of FCRE} the projection $\pi_{\mathcal X}: \mathcal P \to \mathcal X$ is a smooth submersion, i.e.\ $D\pi_{\mathcal X}$ is surjective.
    \end{enumerate}
\end{proposition}
\begin{proof}
    We will show the following implications: $(1) \Leftrightarrow (2)$, $(1) \Rightarrow (3)$, and $(3) \Rightarrow (2)$. 
    For ease of notation, all derivatives in this proof are implicitly evaluated at an arbitrary point $(x,y)$.

    For $(1) \Leftrightarrow (2)$, we use the fact that $\mathcal P$ is a smooth manifold of dimension $\nullity DF$ \cite[Theorem 5.12]{Lee2013}. Applying the rank-nullity theorem gives:
    \begin{align*}
    \dim \mathcal{P} &= \nullity DF = \dim \mathcal X + \dim \mathcal Y - \rank DF\\
    &=\dim \mathcal X + \nullity \pdv{F}{y} + \rank \pdv{F}{y} - \rank DF
    ,\end{align*}
    from which it follows that $(1) \Leftrightarrow (2)$.

    To show $(1) \Rightarrow (3)$, we write the tangent space to $\mathcal P$ as
    $$
    \tspace_{(x,y)}\mathcal P = \ker DF(x,y) = \left\{(\dot{x}, \dot{y}) \,\middle\vert\, \pdv{F(x,y)}{x} \dot x + \pdv{F(x,y)}{x}\dot y = 0 \right\}
    .$$
    Since $\rank DF = \rank \pdv{F}{y}$ and $\img \pdv{F}{y} \subseteq \img DF$, it follows that $V := \img \pdv{F}{y} = \img DF$. Pick any right inverse $\left(\pdv{F}{y}\right)^{RI}: V \to \tspace_y\mathcal Y$ of $\pdv{F}{y}$, i.e.\ $\pdv{F}{y} \circ \left(\pdv{F}{y}\right)^{RI} = \mathrm{Id}_V$. For any $\dot{x} \in \tspace_x \mathcal X$, the vector $\dot{y} = - \left(\pdv{F}{y}\right)^{RI} \pdv{F}{x} \dot{x}$ is well-defined because $\pdv{F}{x} \dot{x} \in V$. It can be verified that $(\dot{x}, \dot{y}) \in \tspace_{(x,y)}\mathcal P$. Hence, $D\pi_{\mathcal X}$ is surjective.

    Finally, we prove $(3) \Rightarrow (2)$. Since $D\pi_{\mathcal X}$ is surjective, it has a right inverse $D\pi_{\mathcal X}^{RI}: \tspace_x \mathcal X \to \tspace_{(x,y)}\mathcal P$. That is, for any $\dot{x} \in \tspace_x \mathcal X$, the vector $D\pi_{\mathcal X}^{RI}(\dot{x})$ is a tuple $(\dot{x}, \dot{y}) \in \ker DF$ for some $\dot{y} \in \tspace_y \mathcal Y$. Hence, $\tspace_{(x,y)}\mathcal P = \ker DF$ contains 
    at least the set 
    $$
    W := \left\{ (\dot x, D\pi_{\mathcal Y}\left(D\pi_{\mathcal X}^{RI} \dot x\right) + v) \,\middle\vert\, \dot{x} \in \tspace_x \mathcal X, v \in \ker \pdv{F}{y}\right\}
    .$$
	It is straightforward to check that $W$ is the image of the injective linear map that takes $(\dot{x}, v) \in \tspace_x \mathcal X \times \ker \pdv{F}{y}$ to $(\dot x, D\pi_{\mathcal Y}\left(D\pi_{\mathcal X}^{RI} \dot x\right) + v)$. Hence, $W$ is a $(\dim \mathcal X + \nullity \pdv{F}{y})$-dimensional linear subspace of $\tspace_{(x,y)}\mathcal P$.
    
    We complete the proof by showing that $\tspace_{(x,y)}\mathcal P \subseteq W$. Pick any $(\dot{x}, \dot{y}) \in \tspace_{(x,y)}\mathcal P$. We have both 
    $$
    \pdv{F}{x}\dot{x} + \pdv{F}{y}\dot{y} = 0
    \quad\text{and}\quad
    \pdv{F}{x}\dot{x} + \pdv{F}{y}\left(D\pi_{\mathcal Y}\left(D\pi_{\mathcal X}^{RI} \dot{x}\right)\right) = 0
    .$$
    By subtracting the latter equation from the former, we see that $\dot{y} - D\pi_{\mathcal Y}\left(D\pi_{\mathcal X}^{RI} \dot{x}\right) \in \ker \pdv{F}{y}$. Thus, $(\dot{x}, \dot{y}) \in W$, and therefore $\tspace_{(x,y)} \mathcal P = W$.
\end{proof}

\begin{remark}
    \label{rmk: FCRE is geometric}
The assumption in \cref{prop: geometric characterisation of FCRE} that $\mathcal P$ is the level set of a map of constant rank is satisfied by any sufficiently small submanifold of $\mathcal X \times \mathcal Y$. More precisely, suppose that $\mathcal P$ is any embedded submanifold of $\mathcal X \times \mathcal Y$. By the local slice criterion \cite[Theorem 5.8]{Lee2013}, every point on $\mathcal P$ has a neighbourhood $\mathcal U$, which is open in $\mathcal X \times \mathcal Y$, such that $\mathcal U \cap \mathcal P$ is the zero set of a map $F: \mathcal U \to \mathbb{R}^{\dim(\mathcal X \times \mathcal Y)}$ of constant rank. Thus, \cref{prop: geometric characterisation of FCRE} can be applied to $\mathcal P \cap \mathcal U$.

The utility of \cref{prop: geometric characterisation of FCRE} is that FCRE problems can be defined either purely in terms of equations (point 1) or purely geometrically (point 3). Both approaches define the same class of problems. 
\end{remark}

An essential part of the proof of \cref{thm: main theorem informal} is that projecting the CREP onto $\mathcal X \times \mathcal Y$ does not introduce singularities if we look at the problem locally.
The argument for this will rely on the following lemma. 
\begin{lemma}
    \label{lemma: image of cst rank map is manifold}
    Let $F: \mathcal M \to \mathcal N$ be a smooth map of constant rank $r$. Then every $p \in \mathcal M$ has a neighbourhood $\mathcal U \subseteq \mathcal M$ such that $F(\mathcal U)$ is an $r$-dimensional embedded submanifold of $\mathcal N$.
\end{lemma}
\begin{proof}
    By \cite[Theorem 4.12]{Lee2013}, there exist charts at $p$ and $F(p)$ such that $F$ is defined in coordinates by $(x^1,\dots,x^{\dim \mathcal M}) \mapsto (x^1,\dots,x^r,0,\dots,0)$ on some neighbourhood $\mathcal{U}'$ of $p$. Let $\pi_r: \mathcal U' \to \mathcal U'$ be the map that sets all but the first $r$ coordinates to zero, so that, locally, $F = F \circ \pi_r$.
    By \cite[Theorem 5.8]{Lee2013}, $\pi_r(\mathcal U')$ is an embedded submanifold of $\mathcal M$. Since $F|_{\pi_{r}(\mathcal U')}$ is an immersion, $F|_{\pi_r(\mathcal U')}$ is a local smooth embedding. The result follows from \cite[Proposition 5.2]{Lee2013}.
\end{proof}

\subsection{Proofs of the main results}
\label{sec: proofs}

Now we can prove \cref{thm: main theorem informal}. In fact, we will prove a stronger statement that gives an expression for the derivative of the map $H$ as well. This is the main theoretical result of the paper.

\begin{theorem}
    \label{thm: deriv soln map CREP}
    Let $F(x,y,z) = c$ be a CREP with $F\colon \mathcal X \times \mathcal Y \times \mathcal Z \to \mathcal W$ and let $(x_0, y_0, z_0)$ be a solution to the CREP. Assume that $\mathcal{Y}$ has a Riemannian metric with an induced distance function $d_{\mathcal Y}$. Then there exist neighbourhoods $\widehat{\mathcal X}$ and $\widehat{\mathcal Z}$ of $x_0$ and $z_0$, respectively, so that the map 
    \begin{align*}
        H\colon \widehat{\mathcal X} &\to \mathcal Y \\
	    x &\mapsto \argmin_{\substack{y \text{ s.t. } \exists z \in \widehat{\mathcal{Z}}\colon\\ F(x,y,z) = c}} d_{\mathcal Y}(y_0, y)
    \end{align*}
    is well-defined and smooth. The differential of $H$ at $x_0$ is the unique linear map $DH(x_0)$ that satisfies the linear system
    \begin{equation}
        \label{eq: formula DH abstract}
        \begin{cases}
            (\dot{x}, DH(x_0)[\dot{x}]) \in D\pi_{\mathcal X \times \mathcal Y}\left[\ker DF\right] \quad\text{for all}\quad \dot{x} \in \tspace_{x_0}\mathcal X\\
            \operatorname{span} DH(x_0) \perp D\pi_{\mathcal Y}\left[\ker \pdv{F}{(y,z)}\right]
        \end{cases}
    \end{equation}
    in which all derivatives are evaluated at $(x_0,y_0,z_0)$ or its projections.
\end{theorem}

The main ideas of the proof are that the projection of the CREP onto $\mathcal X \times \mathcal Y$ satisfies the geometric definition of FCREs (i.e.\ item \ref{itm: geometric definition of FCRE} of \cref{prop: geometric characterisation of FCRE}) and that the theorem is known for FCREs \cite{Dewaele2024}.

\begin{proof}[Proof of \cref{thm: deriv soln map CREP}]
By assumption, the equation $F(x,(y,z)) = c$ defines an FCRE with input $x$ and output $(y,z)$. Define the smooth manifold $\mathcal P := F^{-1}(c)$ and define the projection
$
\pi_{\mathcal X \times \mathcal Y}: \mathcal P \to \mathcal X \times \mathcal Y
.$ At any point $(x,y,z) \in \mathcal P$, it holds that $\tspace_{(x,y,z)}\mathcal P = \ker DF(x,y,z)$. Thus,
\begin{equation}
    \label{eq: ker piXY}
\ker D\pi_{\mathcal X \times \mathcal Y}(x,y,z) = \left\{ (0, 0, \dot z) \in \ker DF(x,y,z) \right\} = \{(0,0)\} \times \ker \pdv{z}F(x,y,z)
,\end{equation}
where the last equality follows from the fact that $DF$ is the sum of all partial derivatives of $F$.
By assumption, the dimension of the right-hand side of \eqref{eq: ker piXY} is independent of $(x,y,z)$. Thus, $\pi_{\mathcal X \times \mathcal Y}$ has constant rank, so that, by \cref{lemma: image of cst rank map is manifold}, the point $(x_0,y_0,z_0)$ has a neighbourhood $\widehat{\mathcal P} \subseteq \mathcal P$ whose projection onto $\mathcal X \times \mathcal Y$ is an embedded submanifold of $\mathcal X \times \mathcal Y$.

Since $F(x,(y,z)) = c$ defines an FCRE, it follows from \cref{prop: geometric characterisation of FCRE} that $\pi_{\mathcal X}: \mathcal P \to \mathcal X$ is a smooth submersion. In other words, the projection of $\tspace_{(x,y,z)} \mathcal P$ onto $\tspace_{x}\mathcal X$ is surjective. Hence, the projection of $\tspace_{(x,y)}\pi_{\mathcal{X} \times \mathcal Y}(\widehat{\mathcal P}) = D\pi_{\mathcal X \times \mathcal Y}[\tspace_{(x,y,z)}\widehat{\mathcal P}]$ is surjective as well. By \cref{prop: geometric characterisation of FCRE,rmk: FCRE is geometric}, $\pi_{\mathcal X \times \mathcal Y}(\widehat{\mathcal P})$ is locally defined by the FCRE $\widetilde{F}(x,y) = 0$ for some unspecified map $\widetilde{F}: \mathcal X \times \mathcal Y \to \mathbb{R}^{\dim(\mathcal X \times \mathcal Y)}$. Since \cref{thm: main theorem informal} is known for FRCEs \cite{Dewaele2024}, the following map is well-defined and smooth within some neighbourhood $\widehat{\mathcal X}$ of $x_0$:
\begin{align*}
    H\colon \widehat{\mathcal X} &\to \mathcal Y \\
    x &\mapsto \argmin_{y\colon \widetilde{F}(x,y) = 0} d_{\mathcal Y}(y_0, y)
    = \argmin_{y\colon (x,y) \in \pi_{\mathcal X \times \mathcal Y}(\widehat{\mathcal P})} d_{\mathcal Y}(y_0, y)
    =
	\argmin_{\substack{y \text{ s.t. } \exists z \in \widehat{\mathcal Z}\colon\\ F(x,y,z) = c}} d_{\mathcal Y}(y_0, y)
,\end{align*}
where $\widehat{\mathcal Z}$ is the projection of $\widehat{\mathcal P}$ onto $\mathcal Z$.
This proves the first assertion.

The derivative of $H$ is given by \eqref{eq: deriv soln map FCRE},
which we will reformulate in terms of the derivatives of $F$. In the following, all derivatives are implicitly evaluated at $(x_0,y_0,z_0)$ or its projections. By the definition of the Moore--Penrose inverse, (i.e.\ $A^\dagger := A|_{(\ker A)^\perp}^{-1}$ for any $A$), \eqref{eq: deriv soln map FCRE} is equivalent to the system
\begin{equation}
    \label{eq: DH system in terms of Ftilde}
    \begin{cases}
        \pdv{\widetilde F}{x} + \pdv{\widetilde F}{y}DH = 0 \\
        \operatorname{span} DH \perp \ker \pdv{\widetilde F}{y}
    \end{cases}
.\end{equation}

The first line of \eqref{eq: DH system in terms of Ftilde} says that $(\dot{x}, DH[\dot{x}]) \in \ker D\widetilde{F}$ for all $\dot{x} \in \tspace_{x_0}\mathcal X$. Since $\pi_{\mathcal X \times \mathcal Y}(\widehat{\mathcal P})$ is a level set of $\widetilde{F}$ and, likewise, $\widehat{\mathcal P}$ is locally a level set of $F$, we have 
$$
\ker D\widetilde{F} = \tspace_{(x_0,y_0)}\pi_{\mathcal X \times \mathcal Y}(\widehat{\mathcal P}) = D\pi_{\mathcal X \times \mathcal Y} \left[\tspace_{(x_0,y_0,z_0)}\widehat{\mathcal P}\right] 
=
D\pi_{\mathcal X \times \mathcal Y} \left[\ker DF\right] 
.$$
Thus, the first lines of \eqref{eq: formula DH abstract} and \eqref{eq: DH system in terms of Ftilde} are equivalent.

For the second condition, we have 
\begin{align*}
    \ker \pdv{\widetilde F}{y} &= \left\{\dot{y} \in \tspace_{y_0} \mathcal Y \,\middle\vert\, (0,\dot{y}) \in \ker D\widetilde{F} \right\} \\
    &=
    \left\{ \dot{y} \in \tspace_{y_0} \mathcal Y \,\middle\vert\, \exists \dot z \in \tspace_{z_0} \mathcal Z \colon (0,\dot{y},\dot{z}) \in \ker DF \right\} \\
    &=
    D\pi_{\mathcal Y}\left[ \ker \pdv{F}{(y,z)} \right]
,\end{align*}
which shows that the second condition in \eqref{eq: formula DH abstract} is equivalent to that in \eqref{eq: DH system in terms of Ftilde}.
\end{proof}

\begin{corollary}
    \label{cor: norm solution map}
    Let $F(x,y,z) = c$ be a CREP with a particular solution $(x_0,y_0,z_0)$ and solution map $H$ as in \cref{thm: deriv soln map CREP}. Then
    $$
    \kappa_{x \mapsto y}[F](x_0,y_0,z_0) = \norm{DH(x_0)}
    $$
    where $\norm{\cdot}$ is the operator norm induced by the metric in $\mathcal X$ and $\mathcal Y$.
\end{corollary}
\begin{proof}
    By definition, $\kappa_{x \mapsto y}[F](x_0,y_0,z_0)$ is the condition number of evaluating $H$ at $x_0$, which is $\norm{DH(x_0)}$ by Rice's characterisation of the condition number \cite{Rice1966}.
\end{proof}

Finally, we show that solving for $y$ in the equation $F(x,y,z) = c$ is at least as well-conditioned as solving for $(y,z)$.

\begin{proof}[Proof of \cref{prop: solving for some variables is easier}]
	Let $H\colon \mathcal X \to \mathcal Y$ and $\widehat{\mathcal Z}$ respectively be 
    the solution map and a neighbourhood of $z_0$ obtained by applying \cref{thm: main theorem informal}
    to the CREP $F(x,y,z) = c$. Likewise, add a dummy variable $\zeta := 0$ to $F$ to    
    define the CREP $\overline{F}(x,(y,z),\zeta) = c$ with input, output, and latent spaces $\mathcal X, \mathcal Y \times \mathcal Z$, and $\{0\}$, respectively, where $\overline{F}(x,(y,z),\zeta) := F(x,y,z)$. Let $\overline{H}\colon \mathcal X \to \mathcal Y \times \mathcal Z$ be its solution map as in \cref{thm: main theorem informal}.

    For $x$ in the domain of both $H$ and $\overline{H}$, we have 
    \begin{align*}
    d_{\mathcal Y \times \mathcal Z}(\overline{H}(x_0), \overline{H}(x))
    &=
    \min_{\substack{(y,z) \in \mathcal Y \times \mathcal Z\\F(x,y,z) = c}} d_{\mathcal Y \times \mathcal Z}((y_0, z_0), (y,z)) \\
    &\ge
    \min_{\substack{(y,z) \in \mathcal Y \times \mathcal Z\\F(x,y,z) = c}}
        d_{\mathcal Y}(y_0, y) \\
    &=
    \min_{\substack{(y,z) \in \mathcal Y \times \widehat{\mathcal Z
    }\\F(x,y,z) = c}}
    d_{\mathcal Y}(y_0, y) \\ 
    &=
    d_{\mathcal Y}(H(x_0), H(x))
    .\end{align*}
    If we divide both sides by $d_{\mathcal X}(x_0, x)$ and take the limit supremum as $x \to x_0$, we obtain $\kappa_{x \mapsto (y,z)}[F](x_0,y_0,z_0) \ge \kappa_{x \mapsto y}[F](x_0,y_0,z_0)$, as desired.

\end{proof}

\section{Numerical computation of the condition number}
\label{sec: compute cond}

This section presents two expressions for the condition number that may be more useful for computational purposes than the abstract \cref{cor: norm solution map}. The first one is a translation of the system \eqref{eq: formula DH abstract} into concrete linear equations. It is stated below in its most general form. When $\mathcal{X},\mathcal{Y},$ and $\mathcal{Z}$ are Euclidean spaces, the statement can be understood more concretely by identifying $\mathcal{X} = \tspace_{x_0}\mathcal X$ with $\mathbb{R}^{\dim \mathcal X}$ and likewise for $\mathcal Y$ and $\mathcal Z$, and the Riemannian metric with the inner product.

\begin{proposition}
    \label{prop: ChElim compute cond}
    Let $F(x,y,z) = c$ be a CREP with $F\colon \mathcal{X} \times \mathcal{Y} \times \mathcal Z \to \mathbb{R}^{N}$, where $\mathcal X$ and $\mathcal Y$ have a Riemannian metric. At any solution $(x_0, y_0, z_0)$, write the partial derivatives $\pdv{F}{x}, \pdv{F}{y}$, and $\pdv{F}{z}$ as matrices $J_x, J_y, J_z$ in coordinates with respect to bases of $\tspace_{x_0}\mathcal X$ and $\tspace_{y_0}\mathcal Y$ that are orthonormal in the chosen Riemannian metrics and an arbitrary basis of $\tspace_{z_0}\mathcal Z$.
    Let $Q$ be a matrix such that $\operatorname{span}Q = (\operatorname{span}J_z)^\perp$ and let $U_y$ and $U_z$ be matrices that satisfy    
    $J_yU_y + J_zU_z = 0$ and $\ker \left[ J_y \quad J_z\right] = \operatorname{span}\left([U_y^T \quad U_z^T]^T\right)$.
    Then the columns of the matrix $A := \begin{bmatrix}
        Q^T J_y \\
        U_y^T
    \end{bmatrix}$ are linearly independent. Furthermore,
    $$
    \kappa_{x \mapsto y}[F](x_0, y_0, z_0) = 
    \norm{
        A^+
        \begin{bmatrix}
            - Q^T J_x \\
            0
        \end{bmatrix}
    }
    ,$$
	where $A^+$ is any left inverse of $A$, such as the Moore--Penrose inverse, and $\norm{\cdot}$ is the operator 2-norm.
\end{proposition}
\begin{proof}
    In the following, we evaluate all derivatives implicitly at $(x_0, y_0, z_0)$ or its projections.
    We start by finding concrete equations for the first constraint in \eqref{eq: formula DH abstract}. For any $(\dot{x}, \dot{y}) \in \tspace_{(x_0, y_0)} (\mathcal X \times \mathcal Y)$, we have
    \begin{align}
    (\dot x, \dot y) \in D\pi_{\mathcal X \times \mathcal Y}\left[\ker DF \right]
    & \Leftrightarrow \pdv{F}{x} \dot{x} + \pdv{F}{y} \dot y + \pdv{F}{z} \dot z = 0 \text{ for some }\dot z \in \tspace_{z_0}\mathcal Z \nonumber \\
    & \Leftrightarrow \pdv{F}{x} \dot{x} + \pdv{F}{y} \dot y \in \mathrm{span}\left(\pdv{F}{z}\right) \nonumber \\
    & \Leftrightarrow Q^T \left(\pdv{F}{x} \dot{x} + \pdv{F}{y} \dot y\right) = 0 \nonumber \\
    & \Leftrightarrow Q^T J_y \hat{\dot{y}} = - Q^T J_x \hat{\dot{x}}
    \label{eq: DH in span Z}
    ,\end{align}
    where $\hat{\dot x}$ and $\hat{\dot y}$ are the coordinates of $\dot x$ and $\dot y$, respectively.
    Similarly, for the second requirement in \eqref{eq: formula DH abstract}, it holds for any $\dot{y} \in \tspace_{y_0}\mathcal Y$ that
    \begin{align*}
        \dot y \perp D\pi_{\mathcal Y}\left[\ker \pdv{F}{(y,z)}\right]
        & \Leftrightarrow \hat{\dot y} \perp [\mathds{I} \quad 0] \ker \left( \left[ J_y \quad J_z \right] \right) \\
        & \Leftrightarrow \hat{\dot y} \perp \operatorname{span} U_y \\
        & \Leftrightarrow U_y^T \hat{\dot y} = 0
    .\end{align*}
    By combining these two observations, we see that \eqref{eq: formula DH abstract} is equivalent to the system 
    \begin{equation}
        \label{eq: concrete system DH}
    A \cdot DH
        =
        \begin{bmatrix}
            -Q^T J_x \\
            0
        \end{bmatrix},
        \quad \text{where} \quad A = 
        \begin{bmatrix}
            Q^T J_y \\
            U_y^T
        \end{bmatrix}
    .\end{equation}
    Since \eqref{eq: formula DH abstract}, or equivalently, \eqref{eq: concrete system DH} has a unique solution, $A$ has full rank and is thereby left-invertible. Hence, the Moore--Penrose inverse of $A$ is a left inverse \cite[\S 5.5.2]{Golub2013}. By \cref{cor: norm solution map}, this concludes the proof.
\end{proof}
\begin{remark}
    \Cref{prop: ChElim compute cond} suggests computing the derivative of the solution map by solving \eqref{eq: concrete system DH}. Although this system has precisely one exact solution, it may be overdetermined. It is possible to reduce the number of equations and keep the same solution. For instance, \eqref{eq: DH in span Z} expresses that a vector that is known to lie in $\operatorname{span} \pdv{F}{(x,y)}$ is also an element of $\operatorname{span} \pdv{F}{z}$. The minimal number of linear equations needed to express this is the codimension of $\operatorname{span} \pdv{F}{z} \cap \operatorname{span} \pdv{F}{(x,y)}$ as a subspace of $\operatorname{span} \pdv{F}{(x,y)}$, but the number of equations used in \eqref{eq: DH in span Z} is the codimension of $\operatorname{span} \pdv{F}{z}$ in $\mathbb{R}^N$. Methods to reduce the number of equations are omitted from this discussion for simplicity.
\end{remark}

Another characterisation of the condition number is given by the following statement. It captures the intuition that $DH(x_0)[\dot{x}]$ gives the smallest possible change to $y$ that solves the CREP when the input is perturbed by $\dot{x}$. Essentially, it uncovers where the defining equations \eqref{eq: formula DH abstract} come from: they are the critical point equations of a convex optimisation problem.
\begin{proposition}
    \label{lemma: DH is soln map of LS}
    Suppose that the equation $F(x,y,z) = c$ is a CREP with a Riemannian input and output space and a particular solution $(x_0,y_0,z_0)$. Then the solution of \eqref{eq: formula DH abstract} is 
    $$
    DH(x_0)\colon
    \dot{x} \mapsto \arg\min_{\dot{y}} \norm{\dot{y}}
    \,\,\text{s.t.}\,\,
    DF(x_0,y_0,z_0)[\dot{x},\dot{y},\dot{z}] = 0
    \,\,\text{for some}\,\, \dot{z} \in \tspace_{z_0}\mathcal{Z}
    .$$
    Consequently,
    $\kappa_{x \mapsto y}[F](x_0,y_0,z_0)$ is the operator norm of $DH(x_0)$.
\end{proposition}
\begin{proof}
    Given any $\dot{x} \in \tspace_{x_0} \mathcal X$, write the set of $(\dot{y},\dot{z})$ that solve the linearisation of $F(x,y,z) = c$ as $L_{\dot x}$. That is,
    $$
    L_{\dot{x}} := \left\{ (\dot{y},\dot{z}) \in \tspace_{y_0} \mathcal Y \times \tspace_{z_0} \mathcal Z
    \,\middle\vert\, 
        DF(x_0,y_0,z_0)[\dot{x},\dot{y},\dot{z}] = 0
    \right\}
    .$$
    Note that $L_{0} = \ker \pdv{F}{(y,z)}$. Thus, \eqref{eq: formula DH abstract} defines $DH(x_0)[\dot{x}]$ as the unique element in $D\pi_{\mathcal Y}\left[L_{\dot x}\right] \cap D\pi_{\mathcal Y}\left[L_0\right]^\perp$.

    Fix any vector $\dot{x} \in \tspace_{x_0} \mathcal X$. The space $L_{\dot x}$ is defined by the linear system $\pdv{F}{(y,z)}[\dot{y},\dot{z}] = -\pdv{F}{x}\dot{x}$. Since $L_0$ is defined by the same equations, but with a different right-hand side, it follows from elementary linear algebra that $L_0$ and $L_{\dot x}$ are parallel. That is, $L_{\dot{x}} = L_0 + v_{\dot x}$ for some $v_{\dot x} \in \tspace_{y_0} \mathcal Y \times \tspace_{z_0} \mathcal Z$. Consequently, $D\pi_{\mathcal Y}\left[L_{\dot x}\right]$ and $D\pi_{\mathcal Y}\left[L_0\right]$ are parallel. 
    Hence, the vector in $D\pi_{\mathcal Y}\left[L_{\dot x}\right]$ with the smallest norm is orthogonal to $D\pi_{\mathcal Y}\left[L_0\right]$ and thereby satisfies the defining equations of $DH(x_0)[\dot{x}]$.

\end{proof}

\section{Application to Tucker decompositions}
\label{sec: Tucker}

Orthogonal Tucker decompositions, simply called \emph{Tucker decompositions} in this paper, are one of several generalisations of matrix decompositions to tensors of arbitrary order \cite{tucker1966some}. In this section, we work out an expression for the condition number of each factor in the decomposition. 

\subsection{Tucker decomposition as a CREP}
First, we review some basic definitions related to Tucker decompositions, which can be found in standard numerical references such as \cite{Kolda2009}.
Any tensor of order $D$ can be written in coordinates as an array $\mathpzc{A} \in \mathbb{R}^{n_1 \times \dots \times n_D}$ and may be expressed as a linear combination of tensor (or Kronecker) products, i.e.\
$$
\mathpzc{A} = \sum_{i=1}^{R} a_{i1} \otimes \dots \otimes a_{iD}
,$$
where $a_{ij} \in \mathbb{R}^{n_j}$ for all $i,j$. Given $D$ matrices $U_1,\dots,U_D$ where $U_j \in \mathbb{R}^{n_j}$, \emph{multilinear multiplication} of $U_1,\dots,U_D$ and $\mathpzc{A}$ is defined as 
$$
(U_1 \otimes \dots \otimes U_D) \mathpzc{A} := \sum_{r=1}^R (U_1 x_{i1}) \otimes \dots \otimes (U_D x_{iD})
.$$
Alternatively, if the map $\mathpzc{A} \mapsto \mathpzc{A}_{(j)}$ denotes reshaping into a matrix of dimensions $n_j \times \prod_{j' \ne j} n_j$, then $(U_1 \otimes \dots \otimes U_D) \cdot \mathpzc{A}$ is the unique tensor $\mathpzc{B}$ satisfying 
$$
\mathpzc{B}_{(j)} = U_j \mathpzc{A}_{(j)} (U_1 \kron \cdots \kron U_{j-1} \kron U_{j+1} \kron \cdots \kron U_D)^T
$$
for any $j = 1,\dots,D$, where $\kron$ is the Kronecker product. For $D = 2$, this definition reduces to the matrix product $(U \otimes V) A = UAV^T$.

A \emph{Tucker decomposition} of a tensor $\mathpzc{X} \in \mathbb{R}^{n_1\times\dots\times n_D}$ is a tuple $(\core, U_1,\dots,U_D)$ where $\core \in \mathbb{R}^{m_1 \times \dots \times m_D}$, the matrices $U_j \in \mathbb{R}^{n_j \times m_j}$ have orthonormal columns, and $\mathpzc{X} = (U_1 \otimes \dots \otimes U_D) \mathpzc{C}$. The minimal size $(m_1,\dots,m_D)$ of $\mathpzc{C}$ in a Tucker decomposition of $\mathpzc{X}$ is called the \emph{multilinear rank} of $\mathpzc{X}$. We study the problem of finding a decomposition given $\mathpzc X$.

The domain of the Tucker decomposition problem has the following well-understood smooth structure \cite{Koch2010}.
The set of real $n_1 \times \dots \times n_D$ tensors of multilinear rank $(m_1,\dots,m_D)$ is a smooth submanifold of $\mathbb{R}^{n_1 \times \dots \times n_D}$. If $m_i = n_i$ for all $i$, we write this set as $\mathbb{R}^{n_1 \times \dots \times n_D}_\star$.
Moreover, for any Tucker decomposition $\mathpzc{X} = (U_1 \otimes \cdots \otimes U_D) \mathpzc{C}$ with $\mathpzc{C} \in \mathbb{R}^{m_1 \times \dots \times m_D}_\star$, we have 
\begin{equation}
    \label{eq: orbits Tucker}
(U_1 \otimes \cdots \otimes U_D) \mathpzc{C}
=
(\tilde{U}_1 \otimes \cdots \otimes \tilde{U}_D) \widetilde{\mathpzc{C}}
\Leftrightarrow 
\mathpzc{C} = (Q_1 \otimes \cdots \otimes Q_D) \widetilde{\mathpzc{C}}
\text{ and }
\tilde{U}_i = U_i Q_i
\end{equation}
for some orthogonal matrices $Q_i \in \mathbb{R}^{m_i \times m_i}$. Hence, a Tucker decomposition where $\mathpzc{C}$ is of minimal size has a fixed number of degrees of freedom. This is the key property that makes it a CREP.
The set of real $n\times m$ matrices with orthonormal columns is known as the \emph{Stiefel manifold} $\mathrm{St}(n, m)$ \cite{Absil2008}.

The Tucker decomposition problem can be modelled as a CREP as follows. Let $\mathcal{X} := \mathbb{R}^{n_1 \times \dots \times n_D}$. Define 
\begin{align*}
\tuck\colon 
\mathbb{R}^{m_1 \times \dots \times m_D}_\star \times \mathrm{St}(n_1,m_1) \times \dots \times \mathrm{St}(n_D,m_D) & \to \mathcal{X} \\
(\core, U_1,\dots,U_D) & \mapsto (U_1 \otimes \dots \otimes U_D) \core
\end{align*}
and 
\begin{equation}
    \label{eq: Tucker eqn with z}
F_{\mathcal T}(\mathpzc{X}, \core, U_1,\dots,U_D) := \mathpzc{X} - \tuck(\core, U_1, \dots, U_D),
\end{equation} where $\mathpzc{X} \in \mathcal X$.
Then the Tucker decomposition problem is equivalent to solving $F_{\mathcal T}(\mathpzc{X}, \core, U_1,\dots,U_D) = 0$.

\subsection{Condition numbers}
With the notation from the previous subsection, we can study the condition number of each factor in the decomposition.
The sensitivity of the factor $U_d$ in the Tucker decomposition with respect to perturbations in $\mathpzc{X}$ is measured by $\kappa_{\mathpzc{X} \mapsto U_d}[F_{\mathcal{T}}]$.
The generic variable names $x,y,$ and $z$ used in the introduction would then refer to $\mathpzc{X}$, $U_d$, and $(\core,U_2,\dots,U_D)$, respectively.
Likewise, the sensitivity of $\core$ is measured by $\kappa_{\mathpzc{X} \mapsto \core}[F_{\mathcal{T}}]$, in which case $y = \core$ and $z = (U_1,\dots,U_D)$.
For simplicity, we omit the subscripts (e.g.,\ $x_0,y_0,z_0$) when referring to the particular solution where the condition number is evaluated.
By the following proposition, the condition number can be computed in terms of the singular values in the \emph{higher-order singular value decomposition} \cite{DeLathauwer2000}.
 
\begin{proposition}
    \label{prop: cond Tucker}
    Let $\mathcal X \in \mathbb{R}^{n_1 \times \dots \times n_D}$ be a tensor with a Tucker decomposition $\mathpzc{X} = \tuck(\core, U_1,\dots,U_D)$.
    Endow the domain of $F_{\mathcal{T}}$ with the Euclidean (i.e.\ Frobenius) norm.    
    Then, for all $d=1,\dots,D$,
    \begin{equation}
        \label{eq: cond Tucker U1}
    \kappa_{\mathpzc{X} \mapsto U_d}[F_{\mathcal{T}}](\mathpzc{X}, \core, U_1,\dots,U_D) = 
    \begin{cases}
        0 & \text{if $U_d$ is square,} \\
        \sigma_{\min}(\core_{(d)})^{-1}  & \text{otherwise,}
    \end{cases}
    \end{equation}
    where $\sigma_{\min}(\core_{(d)})$ is the smallest singular value of the $d$th flattening of $\core$.
    Furthermore,
    \begin{equation}
        \label{eq: cond Tucker core}
    \kappa_{\mathpzc{X} \mapsto \core}[F_{\mathcal{T}}](\mathpzc{X}, \core, U_1,\dots,U_D) = 1
    .\end{equation}
\end{proposition}
\begin{proof}[Proof of \eqref{eq: cond Tucker U1}]
    Since we can permute the arguments of $F_{\mathcal T}$, we can assume without loss of generality that $d = 1$. 
Let $DH(\mathpzc{X})$ be the differential of the canonical solution map and
let $\dot{\mathpzc{X}} \in \tspace_{\mathpzc{X}} \mathcal X$ be any tangent vector.
The equations \eqref{eq: formula DH abstract} that define $DH$ can be specialised to the Tucker decomposition problem as follows:
\begin{equation}
    \label{eq: defining eqns Tucker soln map}
\begin{cases}
    \dot{\mathpzc{X}} = D\tuck[\dot{\core},\dot{U}_1,\dots,\dot{U}_D] \quad\text{for some}\quad \dot{\core},\dot{U}_2,\dots,\dot{U}_D\\
    \dot{U}_1 \perp D\pi_{ \mathrm{St}(n_1,m_1)} \left[ \ker D\tuck \right]
\end{cases}
,\end{equation}
where $\dot{U_1} := DH(\mathpzc{X})[\dot{\mathpzc{X}}]$ and $D\tuck$ is evaluated at $(\core,U_1,\dots,U_D)$. 

	Before solving the system \cref{eq: defining eqns Tucker soln map}, we simplify its second condition.
By \eqref{eq: orbits Tucker}, 
$$
D\pi_{\mathrm{St}(n_1,m_1)} \left[\ker D\tuck\right] = 
\{ 
U_1 \dot{Q} \,\vert\, \dot{Q} \in \tspace_{\mathds{I}} O(m_1) 
\}
\subseteq 
\{ 
U_1 \dot{Q} \,\vert\, \dot{Q} \in \mathbb{R}^{m_1 \times m_1}
\}
,$$
where $O(m_1)$ is the orthogonal group.
Thus, the second constraint in \eqref{eq: defining eqns Tucker soln map} is satisfied on the sufficient (but not necessary) condition that $\dot{U}_1^T U_1 = 0$.

We can solve \eqref{eq: defining eqns Tucker soln map} for $\dot{U}_1$ as a function of $\dot{\mathpzc{X}}$ as follows.
We know from \cite{Koch2010} that every $\dot{\mathpzc{X}} \in \tspace_{\mathpzc{X}}\mathcal X$ admits a unique decomposition of the form 
\begin{align}
    \label{eq: decomposition Tucker tangent}
\dot{\mathpzc{X}} &= (\dot U_1 \otimes \dots \otimes U_D) \core + \dots + (U_1 \otimes \dots \otimes \dot U_D) \core + (U_1 \otimes \dots \otimes U_D) \dot{\core} \\
&=
D\tuck[\dot{\core},\dot{U}_1,\dots,\dot{U}_D] \nonumber
\end{align}
where $\dot{\core} \in \mathbb{R}^{m_1 \times \dots \times m_D}$ and $\dot{U}_d^T U_d = 0$ for all $d$. The factor $\dot{U}_1$ in this decomposition solves \eqref{eq: defining eqns Tucker soln map} given $\dot{\mathpzc{X}}$. Since $DH(\mathpzc{X})$ is the unique linear map that takes $\dot{\mathpzc{X}} \in \tspace_{\mathpzc{X}}\mathcal X$ to the solution $\dot{U}_1$ of \eqref{eq: defining eqns Tucker soln map}, $DH(\mathpzc{X})[\dot{\mathpzc{X}}]$ evaluates to the factor $\dot{U}_1$ in the decomposition \eqref{eq: decomposition Tucker tangent}, for any $\dot{\mathpzc{X}} \in \tspace_{\mathpzc{X}}\mathcal X$.

Assume that $U_1$ is square. Then the only matrix $\dot{U}_1$ that satisfies the constraint $\dot{U}_1^TU_1 = 0$ is the zero matrix. Since $DH(\mathpzc{X})[\dot{\mathpzc{X}}]$ satisfies this constraint for any $\dot{\mathpzc{X}}$, it follows that $DH(\mathpzc{X})$ is the zero map. Thus, the condition number is zero. In the remainder, we assume that $U_1$ is not square, so that the constraint $\dot{U}_1^TU_1 = 0$ is nontrivial.

The operator norm of $DH(\mathpzc{X})$ can be calculated as $\norm{DH(\mathpzc{X})|_{(\ker DH(\mathpzc{X}))^\perp}}$. That is, we can restrict $DH$ to its row space. To determine this space, note that all summands on the right-hand side of \eqref{eq: decomposition Tucker tangent} are pairwise orthogonal in the Euclidean inner product on $\mathbb{R}^{n_1 \times \dots \times n_D}$, since $\dot{U}_d^T U_d = 0$ for all $d$. In addition, $\ker DH(\mathpzc{X})$ consists of all $\dot{\mathpzc{X}}$ such that the first term in \eqref{eq: decomposition Tucker tangent} vanishes. Hence, 
$$
(\ker DH(\mathpzc{X}))^\perp = \left\{ (\dot{U}_1 \otimes U_2 \otimes \dots \otimes U_D) \core \,\middle\vert\, \dot{U}_1^T U_1 = 0 \right\}
,$$
so that
$$
DH(\mathpzc{X})|_{(\ker DH(\mathpzc{X}))^\perp}\colon
(\dot{U}_1 \otimes U_2 \otimes \dots \otimes U_D) \core \mapsto \dot{U}_1
.$$ 
If we represent tensors as their first standard flattening, the inverse of this map is $L: \dot{U}_1 \mapsto \dot{U}_1 \core_{(1)}(U_2 \kron \cdots \kron U_D)^T$ where $\kron$ is the Kronecker product. The singular values of $L$ are the singular values of $\core_{(1)}$. In conclusion:
$$
\norm{DH(\mathpzc{X})} =
\norm{DH(\mathpzc{X})|_{(\ker DH(\mathpzc{X}))^\perp}} = 
1 / \sigma_{m_1}(\core_{(1)})
.$$
\end{proof}
\begin{proof}[Proof of \eqref{eq: cond Tucker core}]
    We use the characterisation of $DH$ from \cref{lemma: DH is soln map of LS}, i.e.\
    $$
    DH(\mathpzc{X})[\dot{\mathpzc{X}}] = \arg\min_{\dot{\core}} \norm{\dot{\core}}
    \,\,\text{s.t.}\,\,
    \dot{\mathpzc{X}} = D\tuck[\dot{\core},\dot{U}_1,\dots,\dot{U}_D]
    \,\,\text{for some}\,\,
    \dot{U}_1,\dots,\dot{U}_D
    .$$
    For any $\dot{\mathpzc{X}} \in \tspace_{\mathpzc{X}}\mathcal X$, the minimum can be estimated by decomposing $\dot{\mathpzc{X}}$ uniquely as $\dot{\mathpzc{X}} = D\tuck[\dot{\core},\dot{U}_1,\dots,\dot{U}_D]$ with $\dot{U}_d^T U_d = 0$ for all $d$, as in \eqref{eq: decomposition Tucker tangent}.
    If we multiply \eqref{eq: decomposition Tucker tangent} on the left by $(U_1^T \otimes \dots \otimes U_D^T)$, all but one term vanish and we obtain 
    $
    \dot{\core} = (U_1^T \otimes \dots \otimes U_D^T) \dot{\mathpzc{X}}
    .$ Thus, in this specific decomposition of $\dot{\mathpzc{X}}$, we have $\norm{\dot{\core}} \le \norm{\dot{\mathpzc{X}}}$. It follows that the operator norm of $DH(\mathpzc{X})$ is at most 1.

    To show that $\norm{DH(\mathpzc{X})} = 1$, it suffices to find a tangent vector $\dot{\mathpzc{X}}$ such that $\norm{DH(\mathpzc{X})[\dot{\mathpzc{X}}]} = \norm{\dot{\mathpzc X}}$. Pick the radial direction $\dot{\mathpzc{X}} := \mathpzc{X} = (U_1 \otimes \dots \otimes U_D) \core$. We can see that $DH(\mathpzc{X})[\dot{\mathpzc{X}}] = \core$ by an argument from the proof of \cite[Proposition 6.2]{Dewaele2024}, which we repeat here. 

    The kernel of $D\tuck$ is the tangent space to the preimage $\tuck^{-1}(\mathpzc{X})$. Since the projection of $\tuck^{-1}(\mathpzc{X})$ onto the first component is the orbit of $\core$ under the action of $O(m_1) \times \dots \times O(m_D)$ (see e.g. \eqref{eq: orbits Tucker}), it is contained in the sphere of radius $\norm{\core}$. Hence, the radial direction $\dot{\core} = \core$ is normal to $D\pi_{\mathbb{R}^{n_1 \times \dots \times n_D}_\star}\left[\ker D\tuck\right]$ and thereby solves the critical point equations \eqref{eq: formula DH abstract}. It follows that $\core = DH(\mathpzc{X})[\dot{\mathpzc{X}}]$ when $\dot{\mathpzc{X}} = \mathpzc{X}$. Since $\norm{\core} = \norm{\mathpzc{X}}$, this completes the proof.

\end{proof}

\Cref{prop: cond Tucker} has the following heuristic explanation. The factor $U_1$ gives a basis for the column space of the flattened tensor $\mathpzc{X}_{(1)} \in \mathbb{R}^{n_1 \times n_2\cdots n_D}$. 
If $U_1$ is square, then $\operatorname{span} (\mathpzc{X}_{(1)}) = \mathbb{R}^{n_1}$. If a perturbed tensor $\widetilde{\mathpzc{X}}$ is sufficiently close to $\mathpzc{X}$, its column space is also $\mathbb{R}^{n_1}$ by the Eckart--Young theorem \cite[Theorem IV.4.18]{Stewart1990}. Therefore, $\widetilde{\mathpzc{X}}$ admits a Tucker decomposition whose first factor is $U_1$.
Alternatively, if $U_1$ is not square, its column space may rotate as $\mathpzc{X}$ is perturbed to $\widetilde{\mathpzc{X}}$. Standard results such as Wedin's $\sin \Theta$-theorem \cite[Theorem V.4.4]{Stewart1990} bound the largest rotation angle in terms of the size of the perturbation to $\widetilde{\mathpzc{X}}$ and $\sigma_{m_1}(\mathpzc{X}_{(1)})^{-1}$. Thus, it is natural to expect the latter quantity to be the condition number, and this expectation is confirmed by \cref{prop: cond Tucker}.

\begin{remark}
    By \cite[Section 14.3.2]{Burgisser2013}, the condition number of computing the image of a matrix $X \in \mathbb{R}^{n_1 \times n_2}$ of rank $m$ is $1 / \sigma_m(X)$.
    Since computing the $U_1$ factor of the Tucker decomposition of a tensor $\mathpzc{X}$ is equivalent to computing an orthonormal basis of $\operatorname{span}(\mathpzc{X}_{(1)})$, one might wonder about the difference between this condition number and the expression in \cref{prop: cond Tucker}. There are two main conceptual differences:
    \begin{enumerate}
        \item The formulation used in \cite{Burgisser2013} quotients out the choice of basis for the image in order to obtain a unique solution. By contrast, \cref{prop: cond Tucker} is about the condition number in \cref{def: cond CREP}, which defines the error in terms of a least-squares projection.
        \item The input space in \cref{prop: cond Tucker} consists of all tensors of multilinear rank $(m_1,\dots,m_D)$. If we flatten a tensor $\mathpzc{X}$ to $\mathpzc{X}_{(1)}$ and apply the approach from \cite{Burgisser2013}, the corresponding condition number takes all perturbations $\widetilde{\mathpzc{X}}$ of $\mathpzc{X}$ into account such that $\widetilde{\mathpzc{X}}_{(1)}$ has rank $m_1$. By contrast, our approach only considers perturbations of $\mathcal{X}$ to $\widetilde{\mathcal{X}}$ that preserve the multilinear rank. That is, our approach has a more constrained input space. 
    \end{enumerate}
\end{remark}

\begin{remark}
\Cref{prop: cond Tucker} fits in perfectly with two earlier results. 
We know from the general result of \cref{prop: solving for some variables is easier} that solving for any one solution variable (e.g., one factor of the decomposition) is at most as ill-conditioned as solving for all solution variables combined. It does not follow in general that there is a single variable that is as ill-conditioned as all variables combined. However, by \cite[Proposition 6.2]{Dewaele2024} and \cref{prop: cond Tucker}, this is the case for the Tucker decomposition: the condition number of the full decomposition is $\max \left(\{1\} \cup \{\sigma_{\min}(\mathpzc{\core}_{(d)})^{-1}\,\vert\, m_d < n_d\} \right)$, i.e.\ the maximum of the condition numbers of the individual factors.
\end{remark}

\section{Conclusion}
\label{sec: conclusion}
In this article, we proposed a condition number for constant-rank elimination problems (CREPs). The condition number estimates the minimal change in the solution variable $y$ for the worst-case infinitesimal perturbation to the input $x$, keeping the latent variable $z$ close to its reference value. By measuring the error this way, we find a lower bound for the condition number of a CREP based on the eliminated variables: if solving for any subset of the variables is ill-conditioned, so must be solving for all variables combined. Since problems with multiple solutions are allowed, the results can be combined with a sensitivity analysis of subsets of the equations defining the CREP.

The condition number of a CREP can be characterised in terms of the partial derivatives of the defining equations. By using this result, we derived a condition number of each factor in an orthogonal Tucker decomposition. The results confirm two simple intuitions. First, the $d$th basis in a Tucker decomposition is ill-conditioned insofar as the $d$th flattening of the tensor has a small singular value. Second, the condition number of the decomposition as a whole equals that of the most sensitive factor.

\section{Acknowledgements}
I thank Nick Vannieuwenhoven, Paul Breiding, Joeri Van der Veken, Mariya Ishteva, Bart Vandereycken, and Carlos Beltr\'an for reviewing the version of this paper that appeared as a chapter in my doctoral dissertation at KU Leuven.

\printbibliography

\end{document}